\documentclass{article}
\usepackage[utf8]{inputenc}
\usepackage{mathtools}
\usepackage{tikz-cd}
\usepackage{amsthm}
\usepackage{amsfonts}
\usepackage{amssymb}
\usepackage{hyperref}
\usepackage{url}
\usepackage{datetime}
\usepackage{cleveref}
\usepackage{hyperref}
\usepackage{xcolor}
\usetikzlibrary{decorations.pathmorphing}

\newcommand{\A}{\mathcal{A}}
\newcommand{\B}{\mathcal{B}}

\newcommand{\E}{\mathcal{E}}
\newcommand{\M}{\mathcal{M}}
\newcommand{\N}{\mathcal{N}}
\newcommand{\both}{\mathrlap{\top}\bot}

\newcommand{\catname}[1]{{\normalfont\textbf{#1}}}
\newcommand{\inj}{\rightarrowtail}
\newcommand{\sur}{\twoheadrightarrow}

\newcommand{\KOmega}{K_0^{\mathrm{Gr}}}
\newcommand{\KnOmega}[2]{K_{#1}^{\mathrm{Gr}}[#2]}

\theoremstyle{definition}
\newtheorem{definition}{Definition}[section]

\newtheorem{remark}[definition]{Remark}

\theoremstyle{plain}
\newtheorem{theorem}[definition]{Theorem}
\newtheorem{lemma}[definition]{Lemma}

\newtheorem{proposition}[definition]{Proposition}

\newdateformat{monthyeardate}{\monthname[\THEMONTH], \THEYEAR}

\title{Combinatorial relative algebraic K-theory}
\author{Jane Turner}
\date{\monthyeardate\today}

\begin{document}

\maketitle

\begin{abstract}
    In a preprint released in 2016, Daniel Grayson introduces a conjectural presentation of the (higher) relative algebraic $K$-groups using purely combinatorial means. In this paper, we will show that this presentation is isomorphic to the classically defined higher relative algebraic $K$-groups, and provide some reasons to believe the conjecture posed in this 2016 preprint to be true.
\end{abstract}

\section*{Introduction}

Algebraic $K$-theory attempts to derive information from mathematical objects, such as the exact category of finitely generated projective modules over a ring, by generalising the process of creating formal inverses to addition formed by direct sum. Given an exact category $\N$, we create a topological space $K\N$ which we call the $K$-theory space of $\N$, which encodes certain algebraic information of $\N$. From this space, we derive an algebraic invariant of $\N$ for each integer $n \geq 0$ called the $n$th $K$-group of $\N$ defined as the $n$th homotopy group of the $K$-theory space of $\N$, namely:
\[ K_n \N := \pi_n K \N. \]

Relative algebraic $K$-theory attempts to study the $K$-theory of exact functors $F: \M \to \N$ between exact categories $\M$ and $\N$. To this end, we define the relative $K$-theory space of $F$ as the homotopy fiber of the map between $K$-theory spaces of $\M$ and $\N$ induced by $F$. In other words:
\[ K[F] := \mathrm{hofib}(K\M \xrightarrow[]{F} K\N). \]
Naturally, echoing the definitions above, we define the $n$th relative $K$-group of $F$ as the $n$th homotopy group of its $K$-theory space
\[ K_n[F] := \pi_n K[F]. \]
Since they are defined using homotopy fibers, the relative $K$-groups of $F$ fit into a long exact sequence with the $K$-groups of $\M$ and $\N$ as shown below:
\begin{equation*}
    \begin{tikzcd}[column sep=small]
        \cdots \arrow[r]
        & K_{n+1} \M \arrow[r]
        & K_{n+1} \N \arrow[r]
        & K_n [F] \arrow[r]
        & K_n \M \arrow[r]
        & K_n \N \arrow[r]
        & \cdots
    \end{tikzcd}
\end{equation*}

In \cite{Gra12}, Grayson establishes a purely algebraic description of the $K$-groups of an exact category $\N$ using so-called $n$-dimensional binary acyclic multicomplexes. We will denote his presentation of the $n$th $K$-group of $\N$ by $K_n^{\mathrm{Gr}}\N$. Later, in \cite{Gra16}, Grayson constructs a conjectural presentation of the $n$th relative $K$-groups of an exact functor $F: \M \to \N$ between exact categories $\M$ and $\N$ in terms of (binary) complexes in $\M$ and $\N$. Grayson then shows that this presentation also sits in a long exact sequence
\begin{equation*}
    \begin{tikzcd}[column sep=tiny]
        \cdots \arrow[r]
        & K_{n+1}^{\mathrm{Gr}} \M \arrow[r]
        & K_{n+1}^{\mathrm{Gr}} \N \arrow[r]
        & \KnOmega{n}{F} \arrow[r]
        & K_n^{\mathrm{Gr}} \M \arrow[r]
        & K_n^{\mathrm{Gr}} \N \arrow[r]
        & \cdots
    \end{tikzcd}
\end{equation*}

In this paper, as our main result, we construct maps
\[ \Phi: \KnOmega{n}{F} \to K_n[F] \]
and prove that these maps are isomorphisms (\Cref{thm:MainResult}). In the process of proving this result, we make explicit the isomorphism $K_1^{\mathrm{Gr}}\N \to K_1 \N$ given in \cite[Corollary~5.10]{Gra12} (\Cref{prop:K1Elements}). In \Cref{lem:NenashevCompatible}, we show that this isomorphism is compatible with the isomorphism $K_1^{\mathrm{Nen}} \N \to K_1 \N$ given in \cite{Nen98} (where here we denote Nenashev's presentation of $K_1$ by $K_1^{\mathrm{Nen}} \N$), via the obvious homommorphism $K_1^{\mathrm{Nen}} \N \to K_1^{\mathrm{Gr}} \N$ shown to be an isomorphism in \cite{KKW18}.

To conclude this introduction, we will provide an overview of the contents of this paper by section. In section $1$ of this paper, we outline various classical constructions of absolute and relative $K$-theory. In particular, we recall the model of the $K$-theory space of an exact category $\N$ given in \cite{GSVW92}, which generalises the $G$-construction of Gillet and Grayson given in \cite{GG87} to the situation of an exact category with weak equivalences in the sense of \cite[Definition~A.1]{Gra12}. We also introduce the relative $K$-theory space $K[F]$ of an exact functor $F: \M \to \N$ as the homotopy fiber $\mathrm{hofib}(K\M \to K\N)$ of the induced map from the $K$-theory space of $\M$ to the $K$-theory space of $\N$, as well as purely combinatorial presentations given by Daniel Grayson in \cite{Gra12} and \cite{Gra16} respectively for $K_1$ of an exact category $\N$ (denoted in this paper by $K_1^{\mathrm{Gr}} \N$) and for $K_0$ of an exact functor $F: \M \to \N$ between exact categories $\M$ and $\N$ (denoted in this paper by $K_0^{\mathrm{Gr}}[F]$). In section $2$, we define a homomorphism $\Phi: K_0^{\mathrm{Gr}}[F] \to K_0[F]$ and prove that this homomorphism is an isomorphism. Over the course of the proof, we provide a concrete description of the isomorphism $K_1^{\mathrm{Gr}} \N \cong K_1 \N$ given in \cite{Gra12}. In section $3$, we recall the presentation of $K_1$ of an exact category $\N$ given by Alexander Nenashev in \cite{Nen98} (denoted in this paper by $K_1^{\mathrm{Nen}} \N$) and show that the isomorphism $K_1^{\mathrm{Nen}} \N \to K_1 \N$ given by Nenashev is compatible with Grayson's isomorphism $K_1^{\mathrm{Gr}} \N \to K_1 \N$. Finally, in section $4$, we recall the full constructions given in \cite{Gra12} and \cite{Gra16} of the $n$th $K$-groups of an exact category $\N$ (denoted by $K_n^{\mathrm{Gr}} \N$) and of an exact functor $F: \M \to \N$ between exact categories $\M$ and $\N$ (denoted in this paper by $K_n^{\mathrm{Gr}}[F]$), and we define homomorphisms $K_n^{\mathrm{Gr}}[F] \to K_n[F]$ and prove that this is an isomorphism.

\section{Constructions of \texorpdfstring{$K$}{}-theory}

Let $\N$ be an exact category. Following \cite{GG87}, we define the simplicial category $G_{\bullet} \N$ as follows. Given some $n \geq 0$, the category $G_n \N$ consists of diagrams in $\N$ of the following form:
\begin{equation*}
    \begin{tikzcd}[column sep=small]
        & A_0 \arrow[rr, rightarrowtail]
        & & A_1 \arrow[rr, rightarrowtail] \arrow[dd, twoheadrightarrow, dashed]
        & & A_2 \arrow[rr, rightarrowtail] \arrow[dd, twoheadrightarrow, dashed]
        & & \cdots \arrow[rr, rightarrowtail]
        & & A_n \arrow[dd, twoheadrightarrow, dashed] \\
        B_0 \arrow[rr, rightarrowtail]
        & & B_1 \arrow[rr, rightarrowtail, crossing over] \arrow[dr, twoheadrightarrow, dashed]
        & & B_2 \arrow[rr, rightarrowtail, crossing over] \arrow[dr, twoheadrightarrow, dashed]
        & & \cdots \arrow[rr, rightarrowtail]
        & & B_n \arrow[dr, twoheadrightarrow, dashed] & \\
        & & & C_{01} \arrow[rr, rightarrowtail]
        & & C_{02} \arrow[rr, rightarrowtail] \arrow[d, twoheadrightarrow]
        & & \cdots \arrow[rr, rightarrowtail]
        & & C_{0n} \arrow[d, twoheadrightarrow] \\
        & & & & & C_{12} \arrow[rr, rightarrowtail] \arrow[rrrrdd, "\ddots" description, phantom]
        & & \cdots \arrow[rr, rightarrowtail]
        & & C_{1n} \arrow[d, twoheadrightarrow]\\
        & & & & & & & & & \vdots \arrow[d, twoheadrightarrow]\\
        & & & & & & & & & C_{nn}
    \end{tikzcd}
\end{equation*}
where $A_i \inj A_j \sur C_{ij}$ and $B_i \inj B_j \sur C_{ij}$ are exact for all $0 \leq i < j \leq n$. In other words, an object of $G_n \N$ is a pair of $n$-step admissible filtrations in $\N$ ($A_i$ and $B_i$) which agree on all successive subquotients, including the data of the choice of said subquotients. A morphism in this category is a map of such diagrams. If $\N$ is equipped with a subcategory $w\N$ of weak equivalences in the sense of \cite{Gra12}, we can equip $G_n \N$ with a subcategory $w G_n \N$ of weak equivalences consisting of morphisms that are level-wise weak equivalences in $w \N$. These subcategories themselves form a simplicial category $wG_{\bullet}\N$, which we can use to create a topological space by taking the nerve to form a bisimplicial set which we then take the diagonal and apply geometric realisation, which we denote as follows:
\[ \lvert wG\N \rvert := \lvert \mathrm{diag}(N_{\bullet} w G_{\bullet} \N) \rvert. \]
This construction is outlined in \cite{GSVW92}, where it is verified that when the weak equivalences are just the isomorphisms in $\N$, the space $\lvert iG\N \rvert$ is homotopy equivalent to the classically defined $K$-theory space of $\N$. For the purposes of this paper, we will use this space as the model for the $K$-theory space, namely:
\[ K\N := \lvert iG\N \rvert. \]
We will now also give meaning to the $n$th $K$-group of $\N$, which we define as the $n$th homotopy group of the $K$-theory space of $\N$, namely:
\[ K_n \N := \pi_n K\N. \]
In \cite[Theorem~2.6]{GSVW92}, they also verify that this construction is homotopy equivalent to the loop space of the space obtained from Waldhausen's $S_{\bullet}$ construction given in \cite{Wal85}, namely:
\[ \lvert wG\N \rvert \simeq \lvert w S_{\bullet} \N \rvert \]
and so we will write $K w\N := \lvert wG\N \rvert$ for the $K$-theory space of any exact category with weak equivalences.
This result allows us to apply the result given in \cite[Theorem~V.2.2]{Wei13} that the $K$-theory of the category of bounded chain complexes with weak equivalences given by quasi-isomorphisms is the same as regular $K$-theory, namely:
\[ K\N \simeq K qC\N. \]

\begin{remark}\label{rem:ComSquare}
    A $1$-simplex of the space $K\N = \lvert iG\N \rvert$ (namely, an element of $N_1 w G_1 \N$) consists of the following data:
    
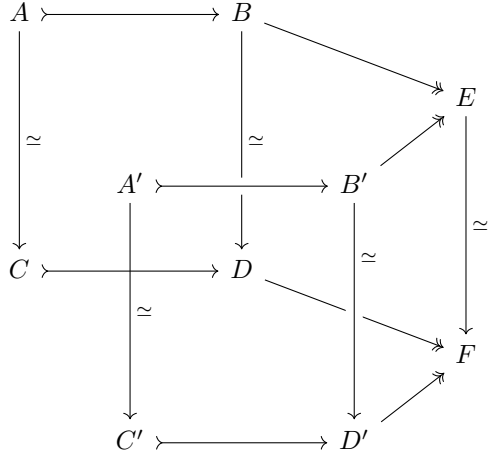
\begin{equation*}
        \begin{tikzcd}
            A \arrow[rr, rightarrowtail] \arrow[ddd, "\simeq"]
            & & B \arrow[drr, twoheadrightarrow] \arrow[ddd, "\simeq"] & & \\
            & & & & E \arrow[ddd, "\simeq"] \\
            & A' \arrow[rr, rightarrowtail, crossing over] \arrow[ddd, "\simeq"]
            & & B' \arrow[ur, twoheadrightarrow] & \\
            C \arrow[rr, rightarrowtail]
            & & D \arrow[drr, twoheadrightarrow] & & \\
            & & & & F \\
            & C' \arrow[rr, rightarrowtail]
            & & D' \arrow[ur, twoheadrightarrow] \arrow[from=uuu, "\simeq" near start, crossing over]
        \end{tikzcd}
    \end{equation*}
    for this $1$-simplex we write $\left( \substack{A \\ A'} \right) \to \left( \substack{D \\ D'} \right)$. Now, using this data, we write
    \begin{equation*}
        \begin{tikzcd}
            \left( \substack{A \\ A'} \right) \arrow[r, rightarrowtail]
            & \left( \substack{B \\ B'} \right)
        \end{tikzcd}
    \end{equation*}
    for the $1$-simplex of $Kw\N$ arising from degenerating an element of $N_0 w G_1 \N$, and we write
    \begin{equation*}
        \begin{tikzcd}
            \left( \substack{A \\ A'} \right) \arrow[r, "\simeq"]
            & \left( \substack{C \\ C'} \right)
        \end{tikzcd}
    \end{equation*}
    for the $1$-simplex in $K w\N$ arising from degenerating an element of $N_1 w G_0 \N$. It turns out that any $1$-simplex is homotopic to a concatenation of one of each of the above two kinds of $1$-simplices. This homotopy can be described by the following diagram:
    \begin{equation*}
        \begin{tikzcd}
            \left( \substack{A \\ A'} \right) \arrow[r, rightarrowtail] \arrow[dr] \arrow[d, "\simeq"]
            & \left( \substack{B \\ B'} \right) \arrow[d, "\simeq"]\\
            \left( \substack{C \\ C'} \right) \arrow[r, rightarrowtail]
            & \left( \substack{D \\ D'} \right)
        \end{tikzcd}
    \end{equation*}
    where the two triangles bound $2$-simplices arising from face and degeneracy maps in the bisimplicial set to obtain non-degenerate simplices in the diagonal. Setting $\left( \substack{A \\ A'} \right) = \left( \substack{B \\ B'} \right) = \left( \substack{C \\ C'} \right)$, we obtain the following diagram:
    \begin{equation*}
        \begin{tikzcd}
            \left( \substack{A \\ A'} \right) \arrow[r, bend left, "\simeq"] \arrow[r, bend right, rightarrowtail]
            & \left( \substack{D \\ D'} \right)
        \end{tikzcd}
    \end{equation*}
    giving us a homotopy between $1$-simplices of different types between two set $0$-simplices.
\end{remark}

Given an exact functor $F: w\M \to v\N$ between exact categories with weak equivalences $w\M$ and $v\N$, we define the relative $K$-theory space of $F$ as the homotopy fiber of the map $Kw\M \to Kv\N$ induced by $F$, which we denote as follows:
\[ K[w\M \xrightarrow[]{F} v\N] := \mathrm{hofib}(Kw\M \xrightarrow[]{KF} Kv\N). \]
Note that this is the same notation as used in \cite{Gra12}, but we use the homotopy fiber rather than the cofiber as we are not in the situation of spectra. Unless explicitly stated, we will operate in the situation where the weak equivalences are just isomorphisms. Similarly to the absolute case, taking the $n$th homotopy group of this space we define the $n$th relative $K$-group of $F$ as
\[ K_n[F] := \pi_n K[F]. \]
Since $K_n[F]$ is defined using homotopy groups of a homotopy fiber, we obtain a long exact sequence
\begin{equation*}
    \begin{tikzcd}[column sep=small]
        \cdots \arrow[r]
        & K_{n+1} \M \arrow[r]
        & K_{n+1} \N \arrow[r]
        & K_n[F] \arrow[r]
        & K_n \M \arrow[r]
        & K_n \N \arrow[r]
        & \cdots
    \end{tikzcd}
\end{equation*}
arising from the homotopy fiber sequence
\begin{equation*}
    \begin{tikzcd}
        K[F] \arrow[r]
        & K\M \arrow[r]
        & K\N
    \end{tikzcd}
\end{equation*}

Recall given an exact category with weak equivalences $w\N$ we have the exact category with weak equivalences $w\E\N$ whose objects consist of short exact sequences in $\M$ and whose arrows are commutative diagrams, where weak equivalences are inherited from $w\N$. We have canonical exact functors $s, t, q: w\E\N \to w\N$ which send a short exact sequence $A \inj B \sur C$ to its source, $A$, its target, $B$, and its quotient, $C$, respectively. We have the following corollary of the additivity theorem applied to the homotopy fiber:

\begin{lemma}\label{lem:RelativeAdditivity}
    Given an exact functor between exact categories with weak equivalences $F: w\A \to w\B$, the induced map
    \[ (s, q): K[w\E\A \xrightarrow[]{F} w\E\B] \to K[w\A^2 \xrightarrow[]{F \times F} w\B^2] \]
    is a homotopy equivalence.
\end{lemma}

\begin{proof}
    This follows from the additivity theorem \cite[Theorem~2.10]{GSVW92} applied to the second and third vertical arrows in the following diagram:
    \begin{equation*}
        \begin{tikzcd}
            K[w\E\A \xrightarrow[]{F} w\E\B] \arrow[r] \arrow[d, "{(s, q)}"]
            & Kw\E\A \arrow[r, "F"] \arrow[d, "{(s, q)}"]
            & Kw\E\B \arrow[d, "{(s, q)}"] \\
            K[w\A^2 \xrightarrow[]{F \times F} w\B^2] \arrow[r]
            & Kw\A^2 \arrow[r, "F \times F"]
            & Kw\B^2
        \end{tikzcd}
    \end{equation*}
\end{proof}

We will now detail the construction given in \cite{Gra12}. Given an exact category $\N$, we denote the full subcategory of $C\N$ of bounded chain complexes that are also acyclic by $C^q \N$. Recall from \cite[Definition~3.1]{Gra12} that a \emph{binary chain complex} is a triple $M = (M_{\bullet}, d, d')$, where $M_{\bullet}$ is a $\mathbb{Z}$-graded object in $\N$, and $d$ and $d'$ are degree $-1$ maps $d, d': M_{\bullet} \to M_{\bullet}$ such that $d^2=d'^2=0$. We write $\top M := (M_{\bullet}, d)$ and $\bot M := (M_{\bullet}, d')$ for the \emph{top} and \emph{bottom} of $M$ respectively, and $\both M$ for the pair $(\top M, \bot M)$ in $C\N^2$. We denote the category of bounded binary chain complexes in $\N$ by $B\N$, and similarly we write $B^q\N$ for the binary chain complexes in $\N$ such that both $\top M$ and $\bot M$ are acyclic. Given a chain complex $M = (M_{\bullet}, d)$ in $C\N$, we write $\Delta M$ for the binary chain complex $(M_{\bullet}, d, d)$, which we call the \emph{diagonal} of $M$. Following Grayson, we denote his presentation of $K_1$ as follows:
\[ K_1^{\mathrm{Gr}}(\N) := \mathrm{coker}(K_0 C^q\N \xrightarrow[]{\Delta} K_0 B^q\N) \]
which was shown in \cite[Corollary~7.4]{Gra12} to be isomorphic to $K_1(\N)$. Since $\Delta$ is split by $\bot$ as functors, we identify this group with
\[ \mathrm{ker}(K_0 B^q\N \xrightarrow[]{\bot} K_0 C^q \N). \]

Now, given an exact functor $F: \M \to \N$, Grayson has the following construction given in \cite{Gra16} to define relative $K$-theory of this functor. We have a category $B[F]$ whose objects consist of triples $(M, N, u)$ with $M \in C\M^2, N \in B\N$ and $u: FM \to \both N$ a quasi-isomorphism. A morphism $(M, N, u) \to (M', N', u')$ in this category is a pair $(\varphi: M \to M', \psi: N \to N')$ such that $\both \psi \circ u = u' \circ F\varphi$. This category can also be equipped with weak equivalences in the sense described in \cite[Definition~A.1]{Gra12}. A morphism $(\varphi, \psi)$ is a weak equivalence if $\varphi$ is a quasi-isomorphism and $\psi$ is an isomorphism. We denote this category of weak equivalences by $pB[F]$. We adopt similar definitions as above for the category $C[F]$, whose objects consist of triples $(M, N, u)$ where $M \in C\M, N \in C\N$ and $u: M \to N$ is a quasi-isomorphism. Given any exact category $\A$ with a subcategory $w\A$ of weak equivalences, we will denote the pair $(\A, w\A)$ by just $w\A$ by abuse of notation, and by $K_0 w\A$, we mean the abelian group generated by isomorphism classes of objects of $\A$, with the usual short exact sequence relation of $K_0$ together with a relation $[A]=[A']$ whenever there is a weak equivalence $A \to A'$. Now, following \cite{Gra16}, we denote Grayson's presentation of $K_0[F]$ as follows:
\[ \KOmega[F] := \mathrm{coker}(K_0 p C[F] \xrightarrow[]{\Delta} K_0 p B[F]) \]
where $\Delta: C[F] \to B[F]$ is the obvious functor defined coordinate-wise.

\section{Grayson's Relative \texorpdfstring{$K_0$}{} is isomorphic to the classically defined group}

Recall our setup of an exact functor $F: \M \to \N$ and we now define a homomorphism
\[ \Phi: \KOmega[F] \to K_0[qC\M \xrightarrow[]{F} qC\N] \cong K_0[F] \]
which we will prove to be an isomorphism.
Recall given a continuous map $f: X \to Y$ between pointed spaces with base points $x_0$ and $y_0$, the \emph{homotopy fiber of $f$} is the topological space
\[ \mathrm{hofib}(f) := \{ (x, \alpha) \in X \times Y^I \mid \alpha(0) = y_0, \alpha(1) = f(x) \}. \]
Given an object $X = (M, N, u)$ of $B[F]$, i.e. a generator of $\KOmega[F]$, we define $\Phi[X]$ to be the connected component of the point
\[ \left( \left( \substack{\top M \\ \bot M} \right), \begin{tikzcd}[column sep=small] \left( \substack{0 \\ 0} \right) \arrow[r, rightarrowtail] & \left( \substack{\top N_{[0, 0]} \\ \bot N_{[0, 0]}} \right) \arrow[r, rightarrowtail] & \left( \substack{\top N_{[0, 1]} \\ \bot N_{[0, 1]}} \right) \arrow[r, rightarrowtail] & \cdots \arrow[r, rightarrowtail] & \left( \substack{\top N \\ \bot N} \right) & \left( \substack{\top FM \\ \bot FM} \right) \arrow[l, "\simeq"'] \end{tikzcd} \right) \]
in $K[qC\M \xrightarrow[]{F} qC\N]$, where the $1$-simplex labeled by ``$\xrightarrow[]{\simeq}$" in the second component is given by $(\top u, \bot u)$ and the other $1$-simplices labeled by ``$\inj$" arise from the na\"ive filtration of $\top N$ and $\bot N$.

\begin{theorem}\label{thm:RelativeK0}
    The assignment
    \[ X \mapsto \Phi(X) \]
    induces a well-defined isomorphism
    \[ \KOmega[F] \to K_0[qC\M \xrightarrow[]{F} qC\N]. \]
\end{theorem}

\begin{proof}
    First, we will show that this is well-defined, and so we need to check that it is compatible with the relations of $\KOmega[F]$. To this end, suppose we have a short exact sequence
    \[ (M', N', u') \inj (M, N, u) \sur (M'', N'', u'') \]
    in $B[F]$. Then we have the point
    \begin{equation*}
        \left(
        \begin{tikzcd}[row sep=2.4em]
            \left( \substack{\top M' \\ \bot M'} \right) \arrow[d, rightarrowtail]\\
            \left( \substack{\top M \\ \bot M} \right) \arrow[d, twoheadrightarrow]\\
            \left( \substack{\top M'' \\ \bot M''} \right)
        \end{tikzcd}
        ,
        \begin{tikzcd}
            \left( \substack{0 \\ 0} \right) \arrow[r, rightarrowtail] \arrow[d, rightarrowtail]
            & \left( \substack{\top N'_{[0, 0]} \\ \bot N_{[0, 0]}} \right) \arrow[r, rightarrowtail] \arrow[d, rightarrowtail]
            & \cdots \arrow[r, rightarrowtail]
            & \left( \substack{\top N'_{[0, n]} \\ \bot N_{[0, n]}} \right) \arrow[d, rightarrowtail]
            & \left( \substack{\top FM' \\ \bot FM'} \right) \arrow[l, "\simeq"'] \arrow[d, rightarrowtail]\\
            \left( \substack{0 \\ 0} \right) \arrow[r, rightarrowtail] \arrow[d, twoheadrightarrow]
            & \left( \substack{\top N_{[0, 0]} \\ \bot N_{[0, 0]}} \right) \arrow[r, rightarrowtail] \arrow[d, twoheadrightarrow]
            & \cdots \arrow[r, rightarrowtail]
            & \left( \substack{\top N_{[0, n]} \\ \bot N_{[0, n]}} \right) \arrow[d, twoheadrightarrow]
            & \left( \substack{\top FM \\ \bot FM} \right) \arrow[l, "\simeq"'] \arrow[d, twoheadrightarrow]\\
            \left( \substack{0 \\ 0} \right) \arrow[r, rightarrowtail]
            & \left( \substack{\top N''_{[0, 0]} \\ \bot N''_{[0, 0]}} \right) \arrow[r, rightarrowtail]
            & \cdots \arrow[r, rightarrowtail]
            & \left( \substack{\top N''_{[0, n]} \\ \bot N''_{[0, n]}} \right)
            & \left( \substack{\top FM'' \\ \bot FM''} \right) \arrow[l, "\simeq"']
        \end{tikzcd}
        \right)
    \end{equation*}
    in $K[q\E C\M \xrightarrow[]{F} q\E C\N]$, where
    \[ n = \mathrm{max}\{ \lvert \mathrm{supp}(N') \rvert, \lvert \mathrm{supp}(N) \rvert, \lvert \mathrm{supp}(N'') \rvert \}. \]
    By \Cref{lem:RelativeAdditivity}, replacing in this point the middle row by the direct sum of the top and bottom rows does not change the connected component it belongs to. Thus, applying the exact functor $\Phi$ yields
    \[ \Phi(M, N, u) = \Phi(M', N', u') + \Phi(M'', N'', u''), \]
    and so $\Phi$ is compatible with short exact sequences.
    
    Now suppose we have a diagonal object $X = (\Delta M, \Delta N, \Delta u)$ of $B[F]$, then due to an easy generalisation \cite[Theorem~3.1]{GG87} (namely that we obtain an involution by swapping the top and bottom filtrations) we have that $\Phi[X]=0$ since it is diagonal in the sense that the two filtrations in every simplex appearing in $\Phi[X]$ are equal, and the fact that the homotopy fiber of a H-map between H-spaces is itself a H-space.
    
    Finally, suppose that we have a map $\varphi: (M, N, u) \to (M', N', u')$ in $pB[F]$, namely the maps $\top M \to \top M'$ and $\bot M \to \bot M'$ are quasi-isomorphisms in $C\M$ and the map $N \to N'$ is an isomorphism in $B\N$. Then $\varphi$ gives us the path
    \begin{equation*}
        \left(
        \begin{tikzcd}
            \left( \substack{\top M \\ \bot M} \right) \arrow[d, "\simeq"]\\
            \left( \substack{\top M' \\ \bot M'} \right)
        \end{tikzcd}
        ,
        \begin{tikzcd}
            \left( \substack{0 \\ 0} \right) \arrow[r, rightarrowtail] \arrow[dr] \arrow[d, "\simeq"]
            & \left( \substack{\top N_{[0, 0]} \\ \bot N_{[0, 0]}} \right) \arrow[r, rightarrowtail] \arrow[dr] \arrow[d, "\simeq"]
            & \cdots \arrow[r, rightarrowtail] \arrow[dr]
            & \left( \substack{\top N_{[0, n]} \\ \bot N_{[0, n]}} \right) \arrow[d, "\simeq"]
            & \left( \substack{\top FM \\ \bot FM} \right) \arrow[l, "\simeq"'] \arrow[dl, "\simeq"'] \arrow[d, "\simeq"]\\
            \left( \substack{0 \\ 0} \right) \arrow[r, rightarrowtail]
            & \left( \substack{\top N'_{[0, 0]} \\ \bot N'_{[0, 0]}} \right) \arrow[r, rightarrowtail]
            & \cdots \arrow[r, rightarrowtail]
            & \left( \substack{\top N'_{[0, n]} \\ \bot N'_{[0, n]}} \right)
            & \left( \substack{\top FM' \\ \bot FM'} \right) \arrow[l, "\simeq"']
        \end{tikzcd}
        \right)
    \end{equation*}
    in $K[qC\M \xrightarrow[]{F} qC\N]$ from the point corresponding to the top row to the point corresponding to the bottom row, where
    \[ n = \mathrm{max} \{ \lvert \mathrm{supp}(N) \rvert, \lvert \mathrm{supp}(N') \rvert \}, \]
    and every triangle in the above diagram bounds a $2$-simplex, and so $\Phi [M, N, u] = \Phi [M', N', u']$, hence $\Phi$ is a well-defined homomorphism.
    
    To show $\Phi$ is an isomorphism by the $5$-lemma, it suffices to show that the following diagram with exact rows and vertical isomorphisms as indicated is commutative:
    \begin{equation*}
        \begin{tikzcd}[column sep=small]
            K_1^{\mathrm{Gr}} \M \arrow[r] \arrow[d, "\cong"]
            & K_1^{\mathrm{Gr}} \N \arrow[r] \arrow[d, "\cong"]
            & \KOmega[F] \arrow[r] \arrow[d, "\Phi"]
            & K_0 \M \arrow[r] \arrow[d, "\cong"]
            & K_0 \N \arrow[d, "\cong"]\\
            K_1 qC\M \arrow[r]
            & K_1 qC\N \arrow[r]
            & K_0[qC\M \xrightarrow[]{F} qC\N] \arrow[r]
            & K_0 qC\M \arrow[r]
            & K_0 qC\N.
        \end{tikzcd}
    \end{equation*}
    The top row of this diagram is defined to be the exact sequence established by Grayson in \cite[Corollary~1.9]{Gra16}, and the bottom row is the exact sequence arising from a homotopy fiber. The two vertical isomorphisms on the left are described by Grayson in \cite[Corollary~7.4]{Gra12}, and the two on the right are given by $[M]-[N] \mapsto [\substack{M \\ N}]$, viewing objects $M$ and $N$ as complexes concentrated in degree $0$. The first and fourth squares commute due to $F$ being an exact functor, and the third square commutes as an element $[M, N, u]$ is sent to $\chi(\top M) - \chi(\bot M)$ in $K_0 \M$ both via the horizontal map and via going around the square with the inverse of the isomorphism $K_0 \M \to K_0 qC\M$, which is given in \cite[Theorem~V.2.2]{Wei13}. It remains to show that the second square commutes. Given an object $N$ of $B^q \N$, we will prove in \Cref{prop:K1Elements} below that the element $[N]$ of $K_1^{\mathrm{Gr}}(\N)$ is sent to the element of $K_1 qC\N$ corresponding to the loop
    \begin{equation*}
        \left(
        \begin{tikzcd}
            \left( \substack{0 \\ 0} \right) \arrow[r, rightarrowtail]
            & \left( \substack{\top N_{[0, 0]} \\ \bot N_{[0, 0]}} \right) \arrow[r, rightarrowtail]
            & \cdots \arrow[r, rightarrowtail]
            & \left( \substack{\top N \\ \bot N} \right)
            & \left( \substack{0 \\ 0} \right) \arrow[l, "\simeq"']
        \end{tikzcd}
        \right)
    \end{equation*}
    where again the $1$-simplices labelled by ``$\rightarrowtail$" arise from the na\"ive filtration of $\top N$ and $\bot N$. This is then sent to the element of $K_0[F]$ arising from the point given by:
    \begin{equation*}
        \left(
        \left( \substack{0 \\ 0} \right),
        \begin{tikzcd}
            \left( \substack{0 \\ 0} \right) \arrow[r, rightarrowtail]
            & \left( \substack{\top N_{[0, 0]} \\ \bot N_{[0, 0]}} \right) \arrow[r, rightarrowtail]
            & \cdots \arrow[r, rightarrowtail]
            & \left( \substack{\top N \\ \bot N} \right)
            & \left( \substack{0 \\ 0} \right) \arrow[l, "\simeq"']
        \end{tikzcd}
        \right)
    \end{equation*}
    On the other hand, $[N]$ is sent to the element $[0, N, 0]$ of $\KOmega[F]$, which $\Phi$ sends to the same element as described above, and so the second square commutes.
\end{proof}

\begin{proposition}\label{prop:K1Elements}
    Given an object $N$ of $B^q \N$, the isomorphism
    \[ K_1^{\mathrm{Gr}}(\N) \to K_1 qC\N \]
    given in \cite{Gra12} sends the element $[N]-[\Delta \bot N]$ to the element $(*)$ corresponding to the loop
    \begin{equation*}
        \left(
        \begin{tikzcd}
            \left( \substack{0 \\ 0} \right) \arrow[r, rightarrowtail]
            & \left( \substack{\top N_{[0, 0]} \\ \bot N_{[0, 0]}} \right) \arrow[r, rightarrowtail]
            & \cdots \arrow[r, rightarrowtail]
            & \left( \substack{\top N \\ \bot N} \right)
            & \left( \substack{0 \\ 0} \right) \arrow[l, "\simeq"']
        \end{tikzcd}
        \right)
    \end{equation*}
    where the arrows labeled by ``$\rightarrowtail$" arise from the na\"ive filtration of $\top N$ and $\bot N$ and the arrow labelled ``$\xleftarrow[]{\simeq}$" arises from the quasi-isomorphism $0 \to N$.
\end{proposition}

Before we proceed with the proof of \cref{prop:K1Elements}, we will present a discussion of the homotopy equivalence given in \cite[Corollary~5.10]{Gra12}, with the caveat that we take as definition $K \Omega \N := K[iB^q\N \xrightarrow[]{\bot} iC^q\N]$. By a standard argument, this is equivalent to the formulation given in \cite{Gra12}, however this formulation will allow us to work entirely with spaces and homotopy fibers rather than spectra and cofibers. Grayson's homotopy equivalence is given by the following chain of equivalences:

Let $C^{\chi}\N$ denote the subcategory of $C\N$ consisting of complexes with vanishing Euler characteristic, which inherits quasi-isomorphisms as weak equivalences. Then \cite[Remark~5.8]{Gra12} tells us that the inclusion $C^{\chi}\N \hookrightarrow C\N$ induces a homotopy fibration sequence:
\[ K qC^{\chi}\N \to KqC\N \to ``K_0\N" \]
where $``K_0\N"$ denotes the Eilenberg-MacLane space whose only nonvanishing homotopy group is $K_0 \N$ at $\pi_0$. After applying $\Omega$ to this sequence, we obtain that the inclusion $C^{\chi}\N \hookrightarrow C\N$ induces a homotopy equivalence
\begin{equation}\label{eq:1}
    \Omega KqC^{\chi}\N \to \Omega KqC\N
\end{equation}

Let $bB\N \supset qB\N$ be the subcategory of weak equivalences of $B\N$ consisting of maps $f$ such that $\bot f$ is a map is $qC\N$. Then $B^b\N$ inherits quasi-isomorphisms as a subcategory of weak equivalences and \cite[Theorem~5.9]{Gra12} shows that the exact functor $\top: qB^b\N \to qC^{\chi}\N$ induces a homotopy equivalence
\begin{equation}\label{eq:2}
    KqB^b\N \to KqC^{\chi}\N.
\end{equation}

Waldhausen's fibration theorem \cite{Wal85} gives a homotopy fiber sequence
\[ KqB^b\N \to KqB\N \to KbB\N. \]
Following \cite[Theorem~4.8]{Gra12}, the map $\bot: KbB\N \to KqC\N$ is a homotopy equivalence. From here, observe the following diagram:
\begin{equation*}
    \begin{tikzcd}
        KqB^b\N \arrow[r] \arrow[d, dashed]
        & KqB\N \arrow[r] \arrow[d, "1"]
        & KbB\N \arrow[d, "\bot"]\\
        K[qB\N \xrightarrow[]{\bot} qC\N] \arrow[r]
        & KqB\N \arrow[r, "\bot"]
        & KqC\N
    \end{tikzcd}
\end{equation*}
The rows of this diagram are homotopy fiber sequences, and the map
\begin{equation}\label{eq:3}
    qB^b\N \to K[qB\N \xrightarrow[]{\bot} qC\N]
\end{equation}
is induced by the middle vertical arrow, and since the middle and right vertical arrows are homotopy equivalences, as is the left vertical arrow.

Here, we diverge in our treatment of this chain of homotopy equivalences, though the map we will arrive at will obviously be the same as the one given in \cite{Gra12}. By Waldhausen's fibration theorem, we have the following pair of fibration sequences:
\begin{align*}
    \Omega KqB\N & \to KiB^q\N \to KiB\N \to KqB\N \\
    \Omega KqC\N & \to KiC^q\N \to KiC\N \to KqC\N
\end{align*}
Assembling these sequences with maps induced by $\bot$, we obtain the following diagram:
\begin{equation*}
    \begin{tikzcd}
        \Omega K[qB\N \xrightarrow[]{\bot} qC\N] \arrow[r, dashed] \arrow[d]
        & K \Omega \N \arrow[r] \arrow[d]
        & * \arrow[d] \\
        \Omega KqB\N \arrow[r, "\partial"] \arrow[d, "\bot"]
        & KiB^q\N \arrow[r] \arrow[d, "\bot"]
        & KiB\N \arrow[d, "\bot"] \\
        \Omega KqC\N \arrow[r, "\partial"]
        & KiC^q\N \arrow[r]
        & KiC\N
    \end{tikzcd}
\end{equation*}
The columns of this diagram are homotopy fibration sequences, and the dashed arrow is hence induced by the map $\partial: \Omega KqB\N \to KiB^q\N$. Since the second and third rows are also homotopy fiber sequences, it follows by the nine lemma that the first row is a homotopy fiber sequence. Hence, the map
\begin{equation}\label{eq:4}
    \Omega K[qB\N \xrightarrow[]{\bot} qC\N] \to K \Omega \N
\end{equation}
is a homotopy equivalence.

Now we will present the proof of \Cref{prop:K1Elements}:
\begin{proof}
    We apply $\pi_0$ to the above discussed chain of homotopy equivalences to obtain a chain of isomorphisms
    \begin{equation*}
        \begin{tikzcd}
            K^{\mathrm{Gr}}_1(\N)
            & K_1 qB^b\N \arrow[r, squiggly] \arrow[l, squiggly]
            & K_1qC\N,
        \end{tikzcd}
    \end{equation*}
    where it is clear that $K_1^{\mathrm{Gr}}(\N) \cong \pi_0 K \Omega \N$ by definition. We will find an element of the group $K_1 qB^b \N$ and show that it is sent on the one hand to the element $[N]-[\Delta \bot N]$ of $K_1^{\mathrm{Gr}}(\N)$ and on the other hand to the element $(*)$ of $K_1 qC\N$, hence proving the statement. To this end, we introduce the exact functor $H: B\N \to B\N$ such that given a binary complex $A$, we have:
    \begin{align*}
        \top HA &= \top A \oplus \bot A [1] \\
        \bot HA &= \mathrm{cone}(\bot A) = \mathrm{cone}(\bot A \xrightarrow[]{1} \bot A).
    \end{align*}
    Then consider the element $(**)$ of $K_1 qB^b\N$ corresponding to the loop
    \begin{equation*}
        \begin{tikzcd}
            \left( \substack{0 \\ 0} \right) \arrow[r, tail]
            & \left( \substack{H(N_{[0, 0]}) \\ H(\Delta \bot N_{[0, 0]})} \right) \arrow[r, tail]
            & \cdots \arrow[r, tail]
            & \left( \substack{HN \\ H \Delta \bot N} \right)
            & \left( \substack{0 \\ 0} \right), \arrow[l, "\simeq"']
        \end{tikzcd}
    \end{equation*}
    where the arrows labelled ``$\rightarrowtail$" arise from applying $H$ to the na\"ive filtration of $N$, and the arrow labelled ``$\xleftarrow[]{\simeq}$" is a weak equivalence by construction of $H$. Now we will apply the following isomorphisms to this element
    \begin{equation*}
        \begin{tikzcd}
            K_1 qB^b\N \arrow[r, "(2)"]
            & K_1 qC^{\chi}\N \arrow[r, "(1)"]
            & K_1 qC\N
        \end{tikzcd}
    \end{equation*}
    when we do, we obtain the element of $K_1 qC\N$ corresponding to the loop
    \begin{equation*}
        \begin{tikzcd}
            \left( \substack{0 \\ 0} \right) \arrow[r, tail]
            & \left( \substack{\top N_{[0, 0]} \oplus \bot N_{[0, 0]}[1] \\ \bot N_{[0, 0]} \oplus N_{[0, 0]}[1]} \right) \arrow[r, tail]
            & \cdots \arrow[r, tail]
            & \left( \substack{\top N \oplus \bot N[1] \\ \bot N \oplus \bot N[1]} \right)
            & \left( \substack{0 \\ 0} \right) \arrow[l, "\simeq"']
        \end{tikzcd}
    \end{equation*}
    We will now show that this element is equal to $(*)$. Observe the following diagram:
    \begin{equation*}
        \begin{tikzcd}
            \left( \substack{0 \\ 0} \right) \arrow[r, tail] \arrow[dr, tail] \arrow[d, tail]
            & \left( \substack{\top N_{[0, 0]} \\ \bot N_{[0, 0]}} \right) \arrow[r, tail] \arrow[dr, tail] \arrow[d, tail]
            & \cdots \arrow[r, tail] \arrow[dr, tail]
            & \left( \substack{\top N \\ \bot N} \right) \arrow[d, tail]
            & \left( \substack{0 \\ 0} \right) \arrow[l, "\simeq"'] \arrow[dl] \arrow[d, tail]\\
            \left( \substack{0 \\ 0} \right) \arrow[r, tail]
            & \left( \substack{\top N_{[0, 0]} \oplus \bot N_{[0, 0]}[1] \\ \bot N_{[0, 0]} \oplus \bot N_{[0, 0]}[1]} \right) \arrow[r, tail]
            & \cdots \arrow[r, tail]
            & \left( \substack{\top N \oplus \bot N[1] \\ \bot N \oplus \bot N[1]} \right)
            & \left( \substack{0 \\ 0} \right) \arrow[l, "\simeq"']
        \end{tikzcd}
    \end{equation*}
    The diagonal $1$-simplices labelled by ``$\rightarrowtail$" are given by the following short exact sequences:
    \begin{align*}
        \top N_{[0, i]} & \inj \top N_{[0, i+1]} \oplus \bot N_{[0, i+1]}[1] \sur N_{i+1}[i+1] \oplus \bot N_{[0, i+1]}[1] \\
        \bot N_{[0, i]} & \inj \bot N_{[0, i+1]} \oplus \bot N_{[0, i+1]}[1] \sur N_{i+1}[i+1] \oplus \bot N_{[0, i+1]}[1].
    \end{align*}
    Using these sequences, one can easily verify that every triangle in the above diagram involving these arrows bound $2$-simplices. The square at the right end of the diagram gains a diagonal arrow and two $2$-simplices in the usual natural way for squares of this kind. Hence, this diagram creates a homotopy between the loops associated with the rows, and hence their corresponding elements of $K_1 qC\N$ are equal.

    Now we return to the element $(**)$ of $K_1 qB^b\N$ and we apply the following isomorphisms:
    \begin{equation*}
        \begin{tikzcd}
            K_1 qB^b\N \arrow[r, "(3)"]
            & K_1[qB\N \xrightarrow[]{\bot} qC\N] \arrow[r, "(4)"]
            & K_1^{\mathrm{Gr}}(\N)
        \end{tikzcd}
    \end{equation*}
    Recalling that $(3)$ is induced by the identity on $K_1 qB\N$ and $(4)$ is induced by $\partial: K_1 qB\N \to K_0 iB^q\N$, we see that $(**)$ is sent to $[HN] - [H \Delta \bot N]$ in $K_1^{\mathrm{Gr}} \N$, since $\partial$ sends the element of $K_1 qB\N$ corresponding to a loop
    \begin{equation*}
        \begin{tikzcd}
            \left( \substack{0 \\ 0} \right) \arrow[r, tail]
            & \left( \substack{A_0 \\ B_0} \right) \arrow[r, tail]
            & \cdots \arrow[r, tail]
            & \left( \substack{A_n \\ B_n} \right)
            & \left( \substack{0 \\ 0} \right) \arrow[l, "\simeq"']
        \end{tikzcd}
    \end{equation*}
    to the connected component of the point $\left( \substack{A_n \\ B_n} \right)$ in $K_0 iB^q\N$. It remains to prove that $[HN]-[H \Delta \bot N]$ is equal to $[N]-[\Delta \bot N]$ in $K_1^{\mathrm{Gr}} \N$. To do this, we have the following short exact sequences
    \begin{align*}
        N \inj HN \sur & \Delta \bot N[1] \\
        \Delta \bot N \inj H \Delta \bot N \sur & \Delta \bot N [1]
    \end{align*}
    which verifies the equality and completes the proof.
\end{proof}

\section{Compatability with work by Nenashev}

In \cite{Nen98}, Nenashev introduced $K_1^{\mathrm{Nen}}\N$, a presentation of $K_1$ of an exact category $\N$ generated by isomorphism classes of binary short exact sequences
\begin{equation*}
    \begin{tikzcd}
        A \arrow[r, shift left, rightarrowtail] \arrow[r, shift right, rightarrowtail]
        & B \arrow[r, shift left, twoheadrightarrow] \arrow[r, shift right, twoheadrightarrow]
        & C
    \end{tikzcd}
\end{equation*}
in $\M$ with a relation $[D] = 0$ whenever $D$ is diagonal (namely, its top differential is equal to its bottom differential). Given a diagram
\begin{equation*}
    \begin{tikzcd}
        A' \arrow[r, shift left, rightarrowtail] \arrow[r, shift right, rightarrowtail] \arrow[d, shift left, rightarrowtail] \arrow[d, shift right, rightarrowtail]
        & A \arrow[r, shift left, twoheadrightarrow] \arrow[r, shift right, twoheadrightarrow] \arrow[d, shift left, rightarrowtail] \arrow[d, shift right, rightarrowtail]
        & A'' \arrow[d, shift left, rightarrowtail] \arrow[d, shift right, rightarrowtail]\\
        B' \arrow[r, shift left, rightarrowtail] \arrow[r, shift right, rightarrowtail] \arrow[d, shift left, twoheadrightarrow] \arrow[d, shift right, twoheadrightarrow]
        & B \arrow[r, shift left, twoheadrightarrow] \arrow[r, shift right, twoheadrightarrow] \arrow[d, shift left, twoheadrightarrow] \arrow[d, shift right, twoheadrightarrow]
        & B'' \arrow[d, shift left, twoheadrightarrow] \arrow[d, shift right, twoheadrightarrow]\\
        C' \arrow[r, shift left, rightarrowtail] \arrow[r, shift right, rightarrowtail]
        & C \arrow[r, shift left, twoheadrightarrow] \arrow[r, shift right, twoheadrightarrow]
        & C''
    \end{tikzcd}
\end{equation*}
in $\M$ where the ``top" arrows horizontally commute with the ``top" arrows vertically (but not necessarily the ``bottom arrows vertically) and vice-versa for ``bottom" arrows, there is a relation
\[ [r_1]-[r_2]+[r_3]=[c_1]-[c_2]+[c_3] \]
where $r_i$ is the $i$th row and $c_j$ is the $j$th column. In \cite{Nen98}, Nenashev provided an isomorphism
\[ K_1^{\mathrm{Nen}}\N \xrightarrow[]{\cong} K_1 \N \]
which sends the element $\left[ \begin{tikzcd} A \arrow[r, shift left, rightarrowtail, "i"] \arrow[r, shift right, rightarrowtail, "j"'] & B \arrow[r, shift left, twoheadrightarrow, "p"] \arrow[r, shift right, twoheadrightarrow, "q"'] & C \end{tikzcd} \right]$ to the element in $K_1 \N$ corresponding to the loop
\begin{equation*}
    \begin{tikzcd}
        \left( \substack{0 \\ 0} \right) \arrow[r, rightarrowtail]
        & \left( \substack{A \\ A} \right) \arrow[r, rightarrowtail, "(i;j)"]
        & \left( \substack{B \\ B} \right)
        & \left( \substack{0 \\ 0} \right). \arrow[l, rightarrowtail]
    \end{tikzcd}
\end{equation*}
where the notation $\begin{tikzcd} \left( \substack{A \\ A} \right) \arrow[r, "(i;j)"] & \left( \substack{B \\ B} \right) \end{tikzcd}$ represents the arrow
\begin{equation*}
    \begin{tikzcd}[row sep=0em]
        A \arrow[r, rightarrowtail, "i"]
        & B \arrow[dr, twoheadrightarrow, "p"] & \\
        & & C \\
        A \arrow[r, rightarrowtail, "j"]
        & B \arrow[ur, twoheadrightarrow, "q"']
    \end{tikzcd}
\end{equation*}

We will now show that this isomorphism is compatible with the isomorphism given by Grayson in \cite{Gra12}
\[ K_1^{\mathrm{Gr}}\N \xrightarrow[]{\cong} K_1 \N \]
in other words:

\begin{lemma}\label{lem:NenashevCompatible}
    The following square commutes:
    \begin{equation*}
        \begin{tikzcd}
            K_1^{\mathrm{Nen}}\N \arrow[r, "\cong"] \arrow[d, "\cong"]
            & K_1 \N \arrow[d, "\cong"]\\
            K_1^{\mathrm{Gr}}\N \arrow[r, "\cong"]
            & K_1 qC\N
        \end{tikzcd}
    \end{equation*}
\end{lemma}

The left vertical isomorphism in the above lemma is induced by the inclusion $B^q_{[0, 2]} \N \hookrightarrow B^q\N$, shown to be an isomorphism in \cite{KKW18}, the right vertical isomorphism is induced by the inclusion $\N \hookrightarrow C\N$ which considers an object as a complex concentrated in degree $0$, and the bottom horizontal arrow sends an element of $K_1^{\mathrm{Gr}} \N$ corresponding to the object $(N, d, d')$ to the element of $K_1 qC\N$ corresponding to the loop
\begin{equation*}
    \left[
    \begin{tikzcd}
        \left( \substack{0 \\ 0} \right) \arrow[r, rightarrowtail]
        & \left( \substack{\top N_{[0, 0]} \\ \bot N_{[0, 0]}} \right) \arrow[r, rightarrowtail]
        & \cdots \arrow[r, rightarrowtail]
        & \left( \substack{\top N \\ \bot N} \right)
        & \left( \substack{0 \\ 0} \right) \arrow[l, "\simeq"']
    \end{tikzcd}
    \right]
\end{equation*}
where the arrows labelled by ``$\inj$" are induced by the na\"ive filtration of $N$, and the arrow labelled by ``$\xleftarrow[]{\simeq}$" is given by the fact $N$ is acyclic. Recall that the bottom horizontal arrow is the one considered in \Cref{prop:K1Elements} and hence the composition $K_1^{\mathrm{Gr}}\N \xrightarrow[]{\cong} K_1 qC\N \xleftarrow[]{\cong} K_1 \N$ is equal to Grayson's isomorphism \cite{Gra12}.

\begin{proof}
    Suppose we have an element of $K_1^{\mathrm{Nen}} \N$ corresponding to the binary short exact sequence
    \begin{equation*}
        \begin{tikzcd}
            A \arrow[r, shift left, rightarrowtail, "i"] \arrow[r, shift right, rightarrowtail, "j"']
            & B \arrow[r, shift left, twoheadrightarrow, "p"] \arrow[r, shift right, twoheadrightarrow, "q"']
            & C
        \end{tikzcd}
    \end{equation*}
    As described above, this is sent to the element of $K_1 \N$ corresponding to the loop
    \begin{equation*}
        \begin{tikzcd}
            \left( \substack{0 \\ 0} \right) \arrow[r, rightarrowtail]
            & \left( \substack{A \\ A} \right) \arrow[r, rightarrowtail, "(i;j)"]
            & \left( \substack{B \\ B} \right)
            & \left( \substack{0 \\ 0} \right) \arrow[l, rightarrowtail]
        \end{tikzcd}
    \end{equation*}
    We also see our element is sent via the composition
    \[ K_1^{\mathrm{Nen}} \N \xrightarrow[]{\cong} K_1^{\mathrm{Gr}} \N \xrightarrow[]{\cong} K_1 qC\N \]
    to the element of $K_1 qC\N$ corresponding to the loop
    \begin{equation*}
        \begin{tikzcd}
            \left( \substack{0 \\ 0} \right) \arrow[r, rightarrowtail]
            & \left( \substack{C \\ C} \right) \arrow[r, rightarrowtail]
            & \left( \substack{B \xrightarrow[]{p} C \\ B \xrightarrow[q]{} C} \right) \arrow[r, rightarrowtail]
            & \left( \substack{A \xrightarrow[]{i} B \xrightarrow[]{p} C \\ A \xrightarrow[j]{} B \xrightarrow[q]{} C} \right)
            & \left( \substack{0 \\ 0} \right) \arrow[l, "\simeq"']
        \end{tikzcd}
    \end{equation*}
    Observe that we have the following diagram in $KqC\N$:
    \begin{equation*}
        \begin{tikzcd}[row sep=huge, column sep=huge]
            \left( \substack{A \xrightarrow[]{i} B \xrightarrow[]{p} C \\ A \xrightarrow[j]{} B \xrightarrow[q]{} C} \right)
            & & \left( \substack{0 \\ 0} \right) \arrow[ll, bend right, "\simeq"'] \arrow[dl, bend right, "\simeq"'] \arrow[dl, bend left, rightarrowtail] \\
            & \left( \substack{A \xrightarrow[]{1} A \xrightarrow[]{} 0 \\ A \xrightarrow[1]{} A \xrightarrow[]{} 0} \right) \arrow[ul, bend right, "\simeq"'] \arrow[ul, bend left, rightarrowtail]
        \end{tikzcd}
    \end{equation*}
    where every space in the diagram can be filled in by $2$-simplices, as explained in \Cref{rem:ComSquare}. This diagram provides a homotopy between the path
    \[ \left( \substack{A \xrightarrow[]{i} B \xrightarrow[]{p} C \\ A \xrightarrow[j]{} B \xrightarrow[q]{} C} \right) \xleftarrow[]{\simeq} \left(  \substack{0 \\ 0} \right) \]
    and the path
    \[ \left( \substack{A \xrightarrow[]{i} B \xrightarrow[]{p} C \\ A \xrightarrow[j]{} B \xrightarrow[q]{} C} \right) \leftarrowtail \left( \substack{A \xrightarrow[]{1} A \xrightarrow[]{} 0 \\ A \xrightarrow[1]{} A \xrightarrow[]{} 0} \right) \leftarrowtail \left( \substack{0 \\ 0} \right). \]
    Thus, our element in $K_1 qC\N$ is equal to the element associated to the loop
    \begin{equation*}
        \begin{tikzcd}[column sep=small]
            \left( \substack{0 \\ 0} \right) \arrow[r, rightarrowtail]
            & \left( \substack{C \\ C} \right) \arrow[r, rightarrowtail]
            & \left( \substack{B \xrightarrow[]{p} C \\ B \xrightarrow[q]{} C} \right) \arrow[r, rightarrowtail]
            & \left( \substack{A \xrightarrow[]{i} B \xrightarrow[]{p} C \\ A \xrightarrow[j]{} B \xrightarrow[q]{} C} \right)
            & \left( \substack{A \xrightarrow[]{1} A \xrightarrow[]{} 0 \\ A \xrightarrow[1]{} A \xrightarrow[]{} 0} \right) \arrow[l, rightarrowtail]
            & \left( \substack{0 \\ 0} \right) \arrow[l, rightarrowtail]
        \end{tikzcd}
    \end{equation*}
    The isomorphism
    \[ K_1 qC\N \to K_1 \N \]
    is induced by the Euler characteristic $\chi: KqC\N \to K\N$. To give an idea of how Euler characteristic works for this representation of the $K$-theory space, below is how Euler characteristic affects $0$-simplices:
    \begin{equation*}
        \left( \substack{\cdots \to M_2 \to M_1 \to M_0 \\ \cdots \to N_2 \to N_1 \to N_0} \right) \mapsto \left( \substack{\cdots \oplus M_2 \oplus N_1 \oplus M_0 \\ \cdots \oplus N_2 \oplus M_1 \oplus N_0} \right).
    \end{equation*}
    We then see that the element of $K_1 \N$ that our loop is sent to is given by the following loop:
    \begin{equation*}
        \begin{tikzcd}
            \left( \substack{0 \\ 0} \right) \arrow[r, rightarrowtail]
            & \left( \substack{C \\ C} \right) \arrow[r, rightarrowtail]
            & \left( \substack{B \oplus C \\ B \oplus C} \right) \arrow[r, rightarrowtail]
            & \left( \substack{A \oplus B \oplus C \\ A \oplus B \oplus C} \right)
            & \left( \substack{A \oplus A \\ A \oplus A} \right) \arrow[l, rightarrowtail]
            & \left( \substack{0 \\ 0} \right) \arrow[l, rightarrowtail]
        \end{tikzcd}
    \end{equation*}
    the first three arrows are diagonal in the sense that both filtrations associated with the $1$-simplices are the same, and so they can be contracted. It is fairly easy to see that the remaining loop is the sum of the following two loops:
    \begin{equation*}
        \begin{tikzcd}
            \left( \substack{0 \\ 0} \right) \arrow[r, rightarrowtail]
            & \left( \substack{A \oplus C \\ A \oplus C} \right)
            & \left( \substack{A \\ A} \right) \arrow[l, rightarrowtail]
            & \left( \substack{0 \\ 0} \right) \arrow[l, rightarrowtail] \\
            \left( \substack{0 \\ 0} \right) \arrow[r, rightarrowtail]
            & \left( \substack{B \\ B} \right)
            & \left( \substack{A \\ A} \right) \arrow[l, rightarrowtail, "\left( \substack{j \\ i} \right)"']
            & \left( \substack{0 \\ 0} \right) \arrow[l, rightarrowtail]
        \end{tikzcd}
    \end{equation*}
    the first of which is obviously equal to $0$ as the middle arrow is diagonal in the sense mentioned above, and the remaining loop is equal to the following loop, as reversing the direction of a loop and exchanging the ``top" and ``bottom" filtrations of the simplices both induce minus signs in the element.
\end{proof}

\section{Higher Relative \texorpdfstring{$K$}{}-groups}

We will now give meaning to a $n$-dimensional multicomplex in $\M$. Firstly, a $1$-dimensional multicomplex is an object of $C\M$, namely, a chain complex in $\M$. An $(n+1)$-dimensional multicomplex is a chain complex in the category $C^n \M$ of $n$-dimensional multicomplexes (i.e. $C^{n+1}\M := C(C^n\M)$. We similarly have the categories $B^n\M$, $(C^q)^n \M$ and $(B^q)^n\M$ of $n$-dimensional binary multicompexes and $n$-dimensional acyclic (binary) multicomplexes.

\begin{definition}
    Given a category $\A$, an object $[f]$ of the arrow category $\A^{\to}$ together with a choice of a map $g$ such that $g \circ f = 1$ is called \emph{split}. We will now give meaning to the category $Q^n \A$ of so-called \emph{split $n$-cubes in $\A$}. Firstly, $Q^0 \A := \A$, i.e. a split $0$-cube in $\A$ is an object of $\A$. A split $(n+1)$-cube in $\A$, namely an object of $Q^{n+1}\A$, is a split object $[A' \to A]$ in the arrow category $(Q^n\A)^{\to}$ of split $n$-cubes in $\A$. A morphism in $Q^{n+1}\A$ is a map $[A' \to A] \to [B' \to B]$ which commutes with the chosen splittings.
    
    Given a functor $G: \A \to \catname{Ab}$ from $\A$ to the category of abelian groups, we can recursively extend $G$ to a functor $G_{n+1}: Q^{n+1} \A \to \catname{Ab}$ from the category of split $(n+1)$-cubes to $\catname{Ab}$. Given a split $(n+1)$-cube $[A' \to A]$ in $Q^{n+1}\A$, we define
    \[ G_{n+1}[A' \to A] := \mathrm{coker}(G_nA' \to G_nA), \]
    or equivalently
    \[ G_{n+1}[A' \to A] = \mathrm{ker}(G_n A \to G_n A') \]
    where the map $G_n A \to G_n A'$ is induced by the choice of splitting of the map $A' \to A$.
\end{definition}

The categories of $n$-dimensional multicomplexes, together with diagonal functors $\Delta$ provide a key example of a split $n$-cube in the category of exact categories, which we will denote by $\Omega^n \M$. Namely, we take as vertices the categories $D_1 D_2 \cdots D_n \M$, where $D_i \in \{ C^q, B^q \}$ for each $i=1, ..., n$, and every morphism is the relevant functor $\Delta$. We take as splittings the corresponding functor $\bot$. This is an arbitrary choice, as we could have instead taken $\top$ as a splitting. As an example, the following is a diagram representing this split $n$-cube in the case of $n=2$:
\begin{equation*}
    \begin{tikzcd}
        (C^q)^2 \M \arrow[r, "\Delta"] \arrow[d, "\Delta"]
        & C^q B^q \M \arrow[d, "\Delta"'] \arrow[l, bend right, "\bot"', dashed] \\
        B^q C^q \M \arrow[r, "\Delta"] \arrow[u, bend left, "\bot", dashed]
        & (B^q)^2 \M \arrow[u, bend right, "\bot"', dashed] \arrow[l, bend left, "\bot", dashed]
    \end{tikzcd}
\end{equation*}
Using the construction discussed above with the functor $G=K_0$ from the category of exact categories to the category of Abelian groups applied to the split $n$-cube of exact categories we have just introduced we obtain the group we will denote by $K_n^{\mathrm{Gr}}(\M)$ introduced in \cite{Gra12}. This group can be generated by isomorphism classes of objects in the category $(B^q)^n\M$ subject to a relation $[M] = [M']+[M'']$ whenever there is a short exact sequence
\[ M' \inj M \sur M'' \]
and a relation $[D] = 0$ whenever $D$ is a diagonal object of $(B^q)^n\M$, namely, whenever $D$ is in the image of any of the above functors labelled by $\Delta$.

Now once again recall our setup of an exact functor $F: \M \to \N$ between exact categories, and consider the object $[F: \M \to \N] =: [F]$ in the arrow category of the category of exact categories associated to $F$. For $n \geq 0$, we denote $\Omega^n [F]$ by the split $n$-cube in the arrow category of the category of exact categories whose vertices are the arrows $[D_1 D_2 \cdots D_n \M \xrightarrow[]{F} D_1 D_2 \cdots D_n \N]$ with $D_i \in \{ C^q, B^q \}$ for $i=1, ..., n$, with morphisms given by $\Delta$ and splittings $\bot$. As before in the absolute case, this is an arbitrary choice where we could have chosen $\top$ instead. As this definition is fairly involved we will give two examples to make it clearer. Firstly, for $n=1$, we will give the following diagram to represent $\Omega^1 [F]$:
\begin{equation*}
    \begin{tikzcd}
        \left[ C^q \M \xrightarrow[]{F} C^q \N \right] \arrow[d, shift right=8, "\Delta"] \arrow[d, shift left=8, "\Delta"'] \\
        \left[ B^q\M \xrightarrow[]{F} B^q \N \right] \arrow[u, bend left, shift left=9, dashed, "\bot"] \arrow[u, bend right, shift right=9, dashed, "\bot"']
    \end{tikzcd}
\end{equation*}
And now for $n=2$, we have the following diagram to represent $\Omega^2[F]$:
\begin{equation*}
    \begin{tikzcd}
        \big[ (C^q)^2 \M \arrow[rr, "F"] \arrow[dr, "\Delta"] \arrow[dd, "\Delta"]
        & & (C^q)^2 \N \big] \arrow[dd, "\Delta" near end] \arrow[dr, "\Delta"] & \\
        & \big[ C^q B^q \M \arrow[rr, crossing over, "F" near start]
        & & C^q B^q \N \big] \arrow[dd, "\Delta"] \\
        \big[ B^q C^q \M \arrow[rr, "F" near end] \arrow[dr, "\Delta"]
        & & B^q C^q \N \big] \arrow[dr, "\Delta"] & \\
        & \big[ (B^q)^2 \M \arrow[rr, "F"] \arrow[from=uu, crossing over, "\Delta" near end]
        & & (B^q)^2 \N \big]
    \end{tikzcd}
\end{equation*}

Considering $\KOmega[F]$, introduced at the end of section $1$, as a functor from $\catname{Exact}^{\to}$ to $\catname{Ab}$, we again recursively extend $K_0 \Omega$ to a functor $Q^n \catname{Exact}^{\to} \to \catname{Ab}$. Applying this functor to the object $\Omega^n[F]$ of $Q^n \catname{Exact}^{\to}$, we obtain the group $\KnOmega{n}{F}$.

This definition is a mouthful, so we will slowly unwrap it for $n=1$. First, we apply $\KOmega[-]$ to the vertices of the split $1$-cube $\Omega^1[F]$. Then, $\KnOmega{1}{F}$ is the cokernel of the induced homomorphism between the two vertices described below:
\begin{equation*}
    \begin{tikzcd}
        \KOmega \left[ C^q \M \xrightarrow[]{F} C^q \N \right] \arrow[d, "\KOmega (\Delta)"] \\
        \KOmega \left[ B^q \M \xrightarrow[]{F} B^q \N \right]
    \end{tikzcd}
\end{equation*}
More explicitly, since $\KOmega[-]$ is also defined using cokernels, we obtain $\KnOmega{1}{F}$ by taking cokernels in both directions of the following commutative diagram:
\begin{equation*}
    \begin{tikzcd}
        K_0 p C \left[ C^q \M \xrightarrow[]{F} C^q \N \right] \arrow[r, "K_0(\Delta)"] \arrow[d, "K_0(\Delta)"]
        & K_0 p B \left[ C^q \M \xrightarrow[]{F} C^q \N \right] \arrow[d, "K_0(\Delta)"] \\
        K_0 p C \left[ B^q \M \xrightarrow[]{F} B^q \N \right] \arrow[r, "K_0(\Delta)"]
        & K_0 p B \left[ B^q \M \xrightarrow[]{F} B^q \N \right]
    \end{tikzcd}
\end{equation*}
Since in all the above arrows $\Delta$ is split by $\bot$, we obtain the same group by taking kernels in both directions of the following commutative diagram:
\begin{equation*}
    \begin{tikzcd}
        K_0 p C \left[ C^q \M \xrightarrow[]{F} C^q \N \right] \arrow[from=r, "K_0(\bot)"'] \arrow[from=d, "K_0(\bot)"']
        & K_0 p B \left[ C^q \M \xrightarrow[]{F} C^q \N \right] \arrow[from=d, "K_0(\bot)"'] \\
        K_0 p C \left[ B^q \M \xrightarrow[]{F} B^q \N \right] \arrow[from=r, "K_0(\bot)"']
        & K_0 p B \left[ B^q \M \xrightarrow[]{F} B^q \N \right]
    \end{tikzcd}
\end{equation*}

\begin{remark}
    Following the above process for higher $n$, we see that the group $\KnOmega{n}{F}$ can be presented in the following way: There is a generator for every isomorphism class of the category $B \left[ (B^q)^n \M \xrightarrow[]{F} (B^q)^n \N \right]$, with the following relations:
    \begin{itemize}
        \item $[Y] = [X] + [Z]$ whenever there is a short exact sequence $X \inj Y \sur Z$
        \item $[D]=0$ for any diagonal object $D$, where diagonal object takes the obvious meaning (after recalling that the functor $\Delta$ can take two meanings, one arising from $C[-] \to B[-]$ and the other from $C^q \to B^q$)
        \item $[X]=[Y]$ whenever there a weak equivalence $X \xrightarrow[]{\simeq} Y$ in the category $pB\left[ (B^q)^n \M \xrightarrow[]{F} (B^q)^n \N \right]$
    \end{itemize}
\end{remark}

Following \cite[Corollary~2.3]{Gra16}, we have a long exact sequence
\begin{equation*}
    \begin{tikzcd}
        K_{n+1}^{\mathrm{Gr}}\M \arrow[r]
        & K_{n+1}^{\mathrm{Gr}}\N \arrow[r]
        & \KnOmega{n}{F} \arrow[r]
        & K_n^{\mathrm{Gr}}\M \arrow[r]
        & K_n^{\mathrm{Gr}}\N
    \end{tikzcd}
\end{equation*}

\begin{theorem}\label{thm:MainResult}
    We have an isomorphism
    \[ \KnOmega{n}{F} \cong K_n[F] \]
    for all $n \geq 0$
\end{theorem}

\begin{proof}
    In \Cref{thm:RelativeK0} we prove the case for $n=0$, so it remains to prove the statement for $n>0$. Recall that (an equivalent formulation of) the definition of $\KnOmega{n}{F}$ is an iterated kernel of the $n$-cube with vertices $\KOmega[D_1 D_2 \cdots D_n F]$ (with $D_1, ..., D_n \in \{ C^q, B^q \}$) and edges induced by $\bot$. Set the space $X$ as the iterated homotopy fiber of the $(n+1)$-cube with vertices $K (D_1 D_2 \cdots D_n \M)$ and $K( D_1 D_2 \cdots D_n \N)$ with edges given by the maps induced by $\bot$ and $F$ in the relevant locations. Then by a standard argument, $\KnOmega{n}{F} \cong \pi_0 X$ since $\KOmega[-] \cong K_0[-]$ as functors. Rearranging the order of taking homotopy fibers, we see that
    \[ \pi_0 X \cong \pi_0 \mathrm{hofib}(K \Omega^n\M \xrightarrow[]{F} K \Omega^n \N). \]
    Due to \cite{Gra12}, we can complete the proof with the following chain of isomorphisms:
    \begin{align*}
        \KnOmega{n}{F} & \cong \pi_0 X \\
        & \cong \pi_0 \mathrm{hofib}(K \Omega^n \M \xrightarrow[]{F} K \Omega^n \N) \\
        & \cong \pi_0 \mathrm{hofib}(\Omega^n K \M \xrightarrow[]{F} \Omega^n K \N) \\
        & \cong \pi_0 \Omega^n K[F] \\
        & \cong K_n[F].
    \end{align*}
\end{proof}

The conjecture \cite[Conjecture~1.6]{Gra16} claims that the following homotopy commutative square is a homotopy pullback square:
\begin{equation*}
    \begin{tikzcd}
        K[pB[F] \xrightarrow[]{\bot} pC[F]] \arrow[r] \arrow[d]
        & K[iB\N \xrightarrow[]{\bot} iC\N] \arrow[d, "{(\both, 1)}"] \\
        K[qC\M^2 \xrightarrow[]{\bot} qC\M] \arrow[r, "F"]
        & K[qC\N^2 \xrightarrow[]{\bot} qC\N]
    \end{tikzcd}
\end{equation*}
where the arrows are induced by the labelled functors, together with the ``forgetful" functors
\begin{align*}
    pB[F] \to qC\M^2, (M, N, u) & \mapsto M \\
    pB[F] \to iB\N, (M, N, u) & \mapsto N,
\end{align*}
and similarly for $C[F]$. Commutativity of this diagram is witnessed by the maps $u: FM \to \both N$ arising from triples $(M, N, u)$ in $B[F]$ (and $C[F]$). One might think that the above theorem proves this conjecture, but various things are missing for a full proof. Most importantly, there is currently no known proof that $K_n^{\mathrm{Gr}}[F] \cong K_n[pB[F] \xrightarrow[]{\bot} pC[F]]$ for $n > 0$. The above diagram gives us a map
\[ K \Omega[F] := K[pB[F] \xrightarrow[]{\bot} pC[F]] \to \mathrm{hoPB} \]
where we write $\mathrm{hoPB}$ for the homotopy pullback of the above diagram. Supposing we had isomorphisms
\[ K_n^{\mathrm{Gr}}[F] \xrightarrow[]{\cong} \pi_n K \Omega[F], \]
in order to prove the conjecture it would remain to show, by Whitehead's theorem, that the following diagram commutes for all $n \geq 0$:
\begin{equation*}
    \begin{tikzcd}
        K_n^{\mathrm{Gr}}[F] \arrow[r, "\cong"] \arrow[d, "\cong"]
        & K_n[F] \arrow[d, "\cong"] \\
        \pi_n K \Omega[F] \arrow[r]
        & \pi_n \mathrm{hoPB}
    \end{tikzcd}
\end{equation*}
where the right vertical isomorphism is induced by the following diagram:
\begin{equation*}
    \begin{tikzcd}
        K[F] \arrow[rr] \arrow[dr, dashed] \arrow[dd]
        & & * \arrow[dr, "\simeq"] \arrow[dd] & \\
        & \mathrm{hoPB} \arrow[rr, crossing over]
        & & K[iB\N \xrightarrow[]{\bot} iC\N] \arrow[dd]\\
        K\M \arrow[rr] \arrow[dr, "\simeq"]
        & & K\N \arrow[dr, "\simeq"] & \\
        & K[qC\M^2 \xrightarrow[]{\bot} qC\M] \arrow[rr] \arrow[from=uu, crossing over]
        & & K[qC\N^2 \xrightarrow[]{\bot} qC\N]
    \end{tikzcd}
\end{equation*}
The front and back faces of this cube are the homotopy pullback squares defining $K[F]$ and $\mathrm{hoPB}$ respectively, and the maps in the bottom square from the back to the front are induced by $\M \to \M^2, M \mapsto (M, 0)$, and similarly for $\N$.

We have for free that $K_0^{\mathrm{Gr}}[F] \cong \pi_0 K \Omega[F]$, and we can indeed prove this commutativity in the case of $n=0$:
\begin{proposition}
    The following diagram commutes:
    \begin{equation*}
        \begin{tikzcd}
            \KOmega[F] \arrow[r, "\Phi", "\cong"'] \arrow[d, "\cong"]
            & K_0[F] \arrow[d, "\cong"]\\
            \pi_0 K \Omega[F] \arrow[r]
            & \pi_0 \mathrm{hoPB}.
        \end{tikzcd}
    \end{equation*}
\end{proposition}

\begin{proof}
    Suppose we have an object $X=(M, N, u)$ of $B[F]$ and observe that we send the element $[X]-[\Delta \bot X]$ of $K_0^{\mathrm{Gr}}[F]$ on the one hand to the element $\left( \left( \substack{X \\ \Delta \bot X} \right), \left( \substack{0 \\ 0} \right) \inj \left( \substack{\bot X \\ \bot X} \right) \right)$ in $K_0[pB[F] \xrightarrow[]{\bot} pC[F]]$, and then to
    \begin{equation*}
        \left(
        \left( \left( \substack{M \\ \Delta \bot M} \right), \left( \substack{0 \\ 0} \right) \rightarrowtail \left( \substack{\bot M \\ \bot M} \right) \right),
        \left(
        \begin{tikzcd}[row sep=small]
            \left( \substack{FM \\ F \Delta \bot M} \right) \arrow[d, "\simeq"]\\
            \left( \substack{\both N \\ \Delta \bot N} \right)
        \end{tikzcd}
        ,
        \begin{tikzcd}[column sep=tiny, row sep=small]
            \left( \substack{0 \\ 0} \right) \arrow[r, rightarrowtail] \arrow[d, "\simeq"] \arrow[dr]
            & \left( \substack{F \bot M \\ F \bot M} \right) \arrow[d, "\simeq"]\\
            \left( \substack{0 \\ 0} \right) \arrow[r, rightarrowtail]
            & \left( \substack{\bot N \\ \bot N} \right)
        \end{tikzcd}
        \right),
        \left(
        \left( \substack{N \\ \Delta \bot N} \right), \left( \substack{0 \\ 0} \right) \rightarrowtail \left( \substack{\bot N \\ \bot N} \right)
        \right)
        \right)
    \end{equation*}
    in $\pi_0 \mathrm{hoPB}$, recalling that the homotopy pullback of a diagram
    \begin{equation*}
        \begin{tikzcd}
            & X \arrow[d, "f"] \\
            Y \arrow[r, "g"]
            & Z
        \end{tikzcd}
    \end{equation*}
    is the space
    \[ \left\{ (x, \varphi, y) \in X \times Z^I \times Y \mid \varphi(0)=f(x), \varphi(1) = g(y) \right\}. \]
    On the other hand, $[X]-[\Delta \bot X]$ is sent to
    \[ \left( \left( \substack{\top M \\ \bot M} \right), \begin{tikzcd}[column sep=small] \left( \substack{0 \\ 0} \right) \arrow[r, rightarrowtail] & \left( \substack{\top N_{[0, 0]} \\ \bot N_{[0, 0]}} \right) \arrow[r, rightarrowtail] & \left( \substack{\top N_{[0, 1]} \\ \bot N_{[0, 1]}} \right) \arrow[r, rightarrowtail] & \cdots \arrow[r, rightarrowtail] & \left( \substack{\top N \\ \bot N} \right) & \left( \substack{\top FM \\ \bot FM} \right) \arrow[l, "\simeq"'] \end{tikzcd} \right) \]
    in $K_0[qC\M \xrightarrow[]{F} qC\N] \cong K_0[F]$, where we recall the arrows labelled ``$\rightarrowtail$" are given by the na\"ive filtration of $\top N$ and $\bot N$, and the arrow labelled by ``$\xleftarrow[]{\simeq}$" is given by the quasi-isomorphism $u: FM \to \both N$. This element is then sent to
    \begin{equation*}
        \left(
        \left( \left( \substack{(\top M, 0) \\ (\bot M, 0)} \right) , c_0 \right),
        \begin{tikzcd}
            \left( \left( \substack{(F \top M, 0) \\ (F \top M, 0)} \right), c_0 \right) \arrow[d, "\simeq"] \\
            \left( \left( \substack{(\top N, 0) \\ (\bot N, 0)} \right), c_0 \right) \\
            \vdots \arrow[u, rightarrowtail] \\
            \left( \left( \substack{(\top N_{[0, 0]}, 0) \\ (\bot N_{[0, 0]}, 0)} \right), c_0 \right) \arrow[u, rightarrowtail] \\
            \left( \left( \substack{0 \\ 0} \right), c_0 \right) \arrow[u, rightarrowtail]
        \end{tikzcd},
        \left( \left( \substack{0 \\ 0} \right), c_0 \right)
        \right)
    \end{equation*}
    in $\pi_0 \mathrm{hoPB}$, where $c_0$ is the constant path at the base point of the relevant space. We have the following path in $K[qC\M^2 \xrightarrow[]{\bot} qC\M]$:
    \begin{equation*}
        \left(
        \begin{tikzcd}
            \left( \substack{(\top M, 0) \\ (\bot M, 0)} \right) \arrow[d, rightarrowtail]\\
            \left( \substack{(\top M, \bot M) \\ (\bot M, \bot M)} \right)
        \end{tikzcd}
        ,
        \begin{tikzcd}
            \left( \substack{0 \\ 0} \right) \arrow[d, rightarrowtail] \arrow[dr, rightarrowtail]\\
            \left( \substack{0 \\ 0} \right) \arrow[r, rightarrowtail]
            & \left( \substack{\bot M \\ \bot M} \right)
        \end{tikzcd}
        \right)
    \end{equation*}
    It is fairly easy to see we have a similar path for every vertex of the path in the middle component, in a way that is compatible with the path, since this operation induces the homotopy equivalence between $K[w\A^2 \xrightarrow[]{\bot} w\A]$ and $Kw\A$ for any exact category with weak equivalences $w\A$. Thus our element of $\pi_0 \mathrm{hoPB}$ is equal to
    \begin{equation*}
        \left(
        \left( \left( \substack{M \\ \Delta \bot M} \right), \left( \substack{0 \\ 0} \right) \rightarrowtail \left( \substack{\bot M \\ \bot M} \right) \right),
        \begin{tikzcd}
            \left( \left( \substack{FM \\ F \Delta \bot M} \right), \left( \substack{0 \\ 0} \right) \rightarrowtail \left( \substack{F \bot M \\ F \bot M} \right) \right) \arrow[d, "\simeq"]\\
            \left( \left( \substack{\both N \\ \Delta \bot N} \right), \left( \substack{0 \\ 0} \right) \rightarrowtail \left( \substack{\bot N \\ \bot N} \right) \right)\\
            \vdots \arrow[u, rightarrowtail]\\
            \left( \left( \substack{\both N_{[0, 0]} \\ \Delta \bot N_{[0, 0]}} \right), \left( \substack{0 \\ 0} \right) \rightarrowtail \left( \substack{\bot N_{[0, 0]} \\ \bot N_{[0, 0]}} \right) \right) \arrow[u, rightarrowtail]\\
            \left( \left( \substack{0 \\ 0} \right), \left( \substack{0 \\ 0} \right) \right)\arrow[u, rightarrowtail]
        \end{tikzcd}
        ,
        \left( \left( \substack{0 \\ 0} \right), \left( \substack{0 \\ 0} \right) \right)
        \right)
    \end{equation*}
    Now, the path in the second component from $\left( \left( \substack{0 \\ 0} \right), \left( \substack{0 \\ 0} \right) \right)$ to $\left( \left( \substack{\both N \\ \Delta \bot N} \right), \left( \substack{0 \\ 0} \right) \rightarrowtail \left( \substack{\bot N \\ \bot N} \right) \right)$ can be lifted to a path in $K[iB\N \xrightarrow[]{\bot} iC\N]$, and so by contracting this path and replacing the point in the third component with $\left( \left( \substack{N \\ \Delta \bot N} \right), \left( \substack{0 \\ 0} \right) \rightarrowtail \left( \substack{\bot N \\ \bot N} \right) \right)$ we verify that our two elements of $\pi_0 \mathrm{hoPB}$ are equal and thus the square above is indeed commutative.
\end{proof}

\bibliographystyle{alpha}
\bibliography{refs}

\begin{thebibliography}{GSVW92}

\bibitem[GG87]{GG87}
Henri Gillet and Daniel~R. Grayson.
\newblock {The loop space of the $Q$-construction}.
\newblock {\em Illinois Journal of Mathematics}, 31(4):574 -- 597, 1987.

\bibitem[Gra12]{Gra12}
Daniel~R. Grayson.
\newblock Algebraic {$K$}-theory via binary complexes.
\newblock {\em J. Amer. Math. Soc.}, 25(4):1149--1167, 2012.

\bibitem[Gra16]{Gra16}
Daniel~R. Grayson.
\newblock Relative algebraic {$K$}-theory by elementary means.
\newblock \url{https://arxiv.org/abs/1310.8644}, 2016.

\bibitem[GSVW92]{GSVW92}
Thomas Gunnarsson, R.~Schw{\"a}nzl, R.~M. Vogt, and F.~Waldhausen.
\newblock An un-delooped version of algebraic {$K$}-theory.
\newblock {\em Journal of Pure and Applied Algebra}, 79(3):255--270, 1992.

\bibitem[KKW20]{KKW18}
Daniel Kasprowski, Bernhard Köck, and Christoph Winges.
\newblock {$K_1$}-groups via binary complexes of fixed length.
\newblock {\em Homology, Homotopy and Applications}, 22(1):203–213, 2020.

\bibitem[Nen98]{Nen98}
Alexander Nenashev.
\newblock {$K_1$} by generators and relations.
\newblock {\em Journal of Pure and Applied Algebra}, 131:195--212, 1998.

\bibitem[Wal85]{Wal85}
Friedhelm Waldhausen.
\newblock Algebraic {$K$}-theory of spaces.
\newblock In Andrew Ranicki, Norman Levitt, and Frank Quinn, editors, {\em Algebraic and Geometric Topology}, pages 318--419, Berlin, Heidelberg, 1985.

\bibitem[Wei13]{Wei13}
C.A. Weibel.
\newblock {\em The $K$-book: An Introduction to Algebraic $K$-theory}.
\newblock Graduate Studies in Mathematics. American Mathematical Society, 2013.

\end{thebibliography}

\end{document}